\documentclass[12pt]{article}
\usepackage{amsmath,amsthm,amssymb,mathtools}
\usepackage{color}
\usepackage{hyperref}
\title{Small ball probability of collision local time for symmetric stable processes}
\author{Minhao Hong and Qian Yu}
\newtheorem{theorem}{Theorem}[section]
\newtheorem{lemma}{Lemma}[section]
\newtheorem{remark}{Remark}[section]
\newtheorem{proposition}{Proposition}[section]
\newtheorem{definition}{Definition}[section]
\newtheorem{corollary}{Corollary}[section]
\numberwithin{equation}{section}
\def\C{\mathbb{C}}
\def\R{\mathbb{R}}
\def\P{\mathbb{P}}
\def\E{{{\mathbb E}\,}}
\def\Re{\mathrm{Re\,}}
\def\Im{\mathrm{Im\,}}
\def\d{\mathrm{d}}
\begin{document}

\maketitle

\begin{abstract}
In this article, the small ball probability is obtained for the collision local time of two independent symmetric $\alpha-$stable processes with parameters $\alpha_1,\alpha_2\in(0,2]$ satisfying $\max\{\alpha_1,\alpha_2\}>1$. The proof is based on obtaining the asymptotic behavior of moment generating function by contour integration.


\end{abstract}
\section{Introduction}
Let $X^{\alpha_1}=\{X^{\alpha_1}_t,t\ge0\}$ and $\widetilde{X}^{\alpha_2}=\{\widetilde{X}^{\alpha_2}_s,s\ge0\}$ be two independent symmetric $\alpha-$stable processes in $\R$ with parameters $\alpha_1,\alpha_2\in(0,2]$, respectively. Throughout this paper we fix a complete probability space $(\Omega,\mathcal{F},P, \mathcal{F}_{t})$ such that the processes considered are well-defined on the space. An $\{\mathcal{F}_{t}\}$-adapted process $X=\{X_{t},t\geq0\}$ with all sample paths in $D[0,\infty]$ is said to be an symmetric $\alpha$-stable process  with $\alpha\in(0,2]$ if for any $t> s\geq0$,
\begin{equation}\label{increment}
	\mathbb{E}[e^{\iota u(X_t-X_s)}|\mathcal{F}_{s}]=e^{-(t-s)|u|^{\alpha}}, u\in\mathbb{R},
\end{equation}
where $X_0=0$. For more about symmetric $\alpha-$stable processes, Section 2 of \cite{zsy2024} and references therein are recommended.

Given $T>0$ the collision local time $\widetilde{L}(T)$ of $X^{\alpha_1}$ and $\widetilde{X}^{\alpha_2}$ is defined formally as 
\begin{align}
	\widetilde{L}(T)=\int_{0}^{T}\delta(X^{\alpha_1}_t-\widetilde{X}^{\alpha_2}_t)\d t,
\end{align}
where $\delta$ is the Dirac function in $\R$. 

Since $\delta$ can be approximated by heat kernel
\[p_\varepsilon(x)=\frac{1}{(2\pi\varepsilon)^{\frac{1}{2}}}e^{-\frac{|x|^2}{2\varepsilon}},x\in\R,\]
collision local time $\widetilde{L}(T)$ can be approximated by 
\[\widetilde{L}_\varepsilon(T)=\int_{0}^{T}p_\varepsilon(X^{\alpha_1}_t-\widetilde{X}^{\alpha_2}_t)\d t\]
as $\varepsilon\downarrow0$. If $\widetilde{L}_\varepsilon(T)$ converges to a random variable in $L^p(p\ge1)$ as $\varepsilon\downarrow0$, it is said that the collision local time $\widetilde{L}(T)$ exists in $L^p$. Collision local time is usually used in the modeling of particle systems (see \cite{bms2024} for instance) and population dynamics (see \cite{ep1998} for instance). For more properties about collision local time, \cite{jw2007}, \cite{g2017}, \cite{syc2012} and references therein are recommended.

\begin{remark}\label{rem 1.1}
	The collision local time of $X^{\alpha_1}$ and $\widetilde{X}^{\alpha_2}$ exists in $L^2$ when $\max\{\alpha_1,\alpha_2\}>1$. Using Fourier transform, there exists
	\begin{align}\label{fourier}
		\widetilde{L}_\varepsilon(T)=\int_{0}^{T}p_\varepsilon(X^{\alpha_1}_t-\widetilde{X}^{\alpha_2}_t)\d t=\frac{1}{2\pi}\int_{0}^{T}\int_{\R}e^{-\frac{\varepsilon}{2}|u|^2}e^{\iota u\cdot (X^{\alpha_1}_t-\widetilde{X}^{\alpha_2}_t)}\d u\d t.
	\end{align}
	Then for any $\varepsilon>0$ and $\eta>0$, 
	\begin{align*}
		\E[\widetilde{L}_\varepsilon(T)-\widetilde{L}_\eta(T)]^2&=\frac{2}{(2\pi)^{2}}\int_{\{0<t_1<t_2<T\}}\int_{\R^{2}}\Big(e^{-\frac{\varepsilon}{2}(|v_1-v_2|^2+|v_2|^2)}-e^{-\frac{\eta}{2}(|v_1-v_2|^2+|v_2|^2)}\Big)	
		\\&\times\exp\Big(-(|v_1|^{\alpha_1}+|v_1|^{\alpha_2})t_1-(|v_2|^{\alpha_1}+|v_2|^{\alpha_2})(t_2-t_1)\Big)\d v_1\d v_2\d t_1\d t_2
	\end{align*}
    from \eqref{increment} and \eqref{fourier}. 
	Since $\big|e^{-\frac{\varepsilon}{2}(|v_1-v_2|^2+|v_2|^2)}-e^{-\frac{\eta}{2}(|v_1-v_2|^2+|v_2|^2)}\big|\le1$ and 
	\begin{equation*}
		\begin{split}
			&\int_{\{0<t_1<t_2<T\}}\int_{\R^{2}}\exp\Big(-(|v_1|^{\alpha_1}+|v_1|^{\alpha_2})t_1-(|v_2|^{\alpha_1}+|v_2|^{\alpha_2})(t_2-t_1)\Big)\d v_1\d v_2\d t_1\d t_2\\&\le\int_{\{0<t_1<t_2<T\}}t_1^{-\frac{1}{\max\{\alpha_1,\alpha_2\}}}(t_2-t_1)^{-\frac{1}{\max\{\alpha_1,\alpha_2\}}}\d t_1\d t_2\cdot\Big(\int_{\R}e^{-|v|^{\max\{\alpha_1,\alpha_2\}}}\d v\Big)^2<+\infty,
		\end{split}
	\end{equation*}
	there exists $\lim\limits_{\varepsilon,\eta\downarrow0}\E[\widetilde{L}_\varepsilon(T)-\widetilde{L}_\eta(T)]^2=0$ from the dominated convergence theorem, which means $\widetilde{L}_\varepsilon(T)\stackrel{L^2}{\to}\widetilde{L}(T)$ as $\varepsilon\downarrow0$. Moreover, $\widetilde{L}(T)$ is a non-negative random variable because of the fact $\widetilde{L}_\varepsilon(T)\ge0$ at this time.
	\end{remark}
    
In this article, the small ball probability of $\widetilde{L}(T)$, the collision local time for two independent symmetric $\alpha-$stable processes will be considered. Small ball probability, concerned with the asymptotic decay of $\mathbb{P}(Z \leq \varepsilon)$ as $\varepsilon \downarrow 0$ for a non-negative random variable $Z$, constitute a fundamental problem in probability theory and functional analysis. \cite{ls2001} is a survey of general development in this direction. Beyond their intrinsic mathematical interest in small deviation theory and embedding theorems, such estimates have profound implications in statistical physics (e.g., polymer models and critical phenomena), quantum field theory (via path integral regularization), and even modern machine learning (through concentration of measure in high dimensions). When $Z$ represents a local time or its variants, the small ball probability quantifies the likelihood of rare events such as ``almost no interaction'' between stochastic particles---a scenario of direct relevance to systems exhibiting weak coupling or sparse collisions.

In the classical Gaussian setting where processes are standard Brownian motions ($\alpha= 2$), the small ball behavior of local times and its variants has been investigated by some researchers. \cite{cs2025} shows that the small ball probability of intersection local time for Brownian motions satisfies
\begin{equation}\label{power}
c\varepsilon^{\frac23}\le P\Big(\int_0^1\int_0^1\delta(B_r-\tilde{B}_s)\d r\d s\le\varepsilon\Big)\le C\varepsilon^{\frac23}
\end{equation}
for some positive constants $c$ and $C$, where $B$ and $\tilde{B}$ are independent $1-$dimensional Brownian motions. The small ball probability of self-intersection local time of Brownian motion satisfies
\begin{equation}\label{exponent}
\lim_{\varepsilon\to0^+}\varepsilon^2\log P\Big(\int_0^1\int_0^1\delta(B_r-B_s)\d r\d s\le\varepsilon\Big)=-c
\end{equation}
for some positive constant $c$, which comes from Proposition 1 of \cite{khh1997} and the de Bruijn's Tauberian theorem (Theorem 4.12.9 of \cite{bgt1989}).

%In the classical Gaussian setting where both processes are standard Brownian motions ($\alpha_1 = \alpha_2 = 2$), the small ball behavior of collision local times has been thoroughly investigated. Seminal works by LeGall(1987) and Bass and Khoshnevisan (1993) leveraged the strong Markov property and potential-theoretic techniques to derive precise asymptotics. These results not only highlight the regenerative structure unique to Brownian paths but also serve as a benchmark for non-Gaussian generalizations.

However, many real-world diffusion phenomena—ranging from turbulent flows and financial asset returns to anomalous transport in disordered media—exhibit heavy-tailed increments and long-range dependence that defy Gaussian modeling. Symmetric $\alpha$-stable processes ($0 < \alpha < 2$) provide a natural and flexible alternative, capturing jump discontinuities and self-similarity inherent in such systems. Despite their irregular sample paths, these processes offer a more realistic framework for studying intermittent interactions or rare encounters in complex environments.

This article extends the small ball analysis from the Gaussian paradigm to the non-Gaussian symmetric stable processes. Specifically, we investigate the asymptotic behavior of $\mathbb{P}(\widetilde{L}(T) \leq \varepsilon)$ as $\varepsilon \downarrow 0$ under the standing assumption $\max\{\alpha_1, \alpha_2\} > 1$, which ensures that $\widetilde{L}(T)$ is well-defined in $L^2$ (see Remark \ref{rem 1.1}).

A central contribution of this work lies in the development of a novel analytical framework based on complex contour integration. Traditional approaches relying on heat kernel estimates or martingale methods face significant obstacles in the non-Gaussian setting due to the absence of explicit transition densities and path continuity. To overcome this, we introduce a carefully designed contour $\gamma(R, 3\pi/4)$ in the complex plane and exploit the integral representation of the reciprocal Gamma function (Lemma \ref{reciprocal}) to transform the moment generating function $\mathbb{E}[e^{-\lambda \widetilde{L}(T)}]$ into a single-variable complex integral. This transformation not only streamlines the asymptotic analysis as $\lambda \to \infty$ but also enables an explicit evaluation of the limiting constant governing the small ball probability.

The methodological innovation here is twofold: first, it demonstrates how probabilistic questions about singular functionals of stable processes can be recast as problems in complex analysis; second, it opens the door to a broader class of applications. The same contour-integral technique could potentially address small ball estimates for multipoint intersection local times of multiple stable particles, collision functionals of fractional Brownian motion, or even support measures of solutions to stochastic partial differential equations driven by L\'{e}vy noise. Thus, beyond resolving a specific open problem, our approach offers a new analytic toolkit for fine-scale probabilistic analysis in non-Gaussian settings.

To connect the asymptotics of the Laplace transform to the small ball probability, we employ the following Tauberian-type result:
\begin{proposition}
Let $Z$ be a non-negative random variable. If $\lim_{\lambda \to \infty} \lambda \mathbb{E}[e^{-\lambda Z}] = C \in (0,\infty)$, then
\[
\lim_{\varepsilon \downarrow 0} \varepsilon^{-1} \mathbb{P}(Z \leq \varepsilon) = C.
\]
\end{proposition}

For clarity and logical flow, we have structured the proof of our main theorem into three distinct stages. In Step I, we derive an explicit expression for the moments $\mathbb{E}[\widetilde{L}(T)^m]$ via Fourier analysis and coordinate transformations. Step II then lifts this moment sequence to the Laplace domain and, crucially, invokes Lemma~2.3 to convert the resulting power series into a contour integral involving an auxiliary function $\Phi(z)$. Finally, Step III extracts the large-$\lambda$ asymptotics of $\lambda \mathbb{E}[e^{-\lambda \widetilde{L}(T)}]$ by analyzing the analytic properties of $\Phi(z)$ along the contour—thereby completing the bridge to the small ball estimate via the following Proposition \ref{Tauberian}, which is a special case of Karamata Tauberian theorem (Theorem 1.7.1 of \cite{khh1997}) and the proof of Proposition \eqref{Tauberian} is proved in the last section of this article.

\begin{proposition}\label{Tauberian}
	Assume that  $X$ is a non-negative random variable. If there exist constants $\alpha>0$ and $A\ge0$ such that 
	\begin{align}\label{tauberian}
		\lim\limits_{\lambda\to+\infty}\lambda^{\alpha}Ee^{-\lambda X}=A,
	\end{align}
	then 
	\[\lim\limits_{\varepsilon\downarrow0}\varepsilon^{-\alpha}P(X\le\varepsilon)=\frac{A}{\Gamma(1+\alpha)}.\]
\end{proposition}

Each technical lemma (e.g., Lemmas \ref{argument}--\ref{integralofPhi}) is crafted to support a specific analytical requirement in this pipeline, such as justifying contour deformation or controlling integrand decay.

 Our main result is the following.
 
\begin{theorem}\label{smalldeviation}
	Assume that $\alpha_1,\alpha_2\in(0,2]$ satisfy $\max\{\alpha_1,\alpha_2\}>1$. Then 
	\begin{equation*}
		\begin{split}
			&\lim\limits_{\varepsilon\downarrow0}\varepsilon^{-1}P(\widetilde{L}(T)\le\varepsilon)\\&=\left\{\begin{array}{ll}
				\displaystyle\frac{2}{T}\int_0^{+\infty}\frac{\Im\big\{ e^{\iota\frac{r}{\sqrt{2}}}\Phi(re^{-\iota\frac{3\pi}{4}})\big\}}{|\Phi(re^{\iota\frac{3\pi}{4}})|^2}\cdot\frac{e^{-\frac{1}{\sqrt{2}}r}}{r}\d r+\frac{3\pi}{2\int_{\R}\frac{1}{|v|^{\alpha_1}+|v|^{\alpha_2}}\d v}, & \min\{\alpha_1,\alpha_2\}<1, \\\\\displaystyle\frac{2}{T}\int_0^{+\infty}\frac{\Im\big\{ e^{\iota\frac{r}{\sqrt{2}}}\Phi(re^{-\iota\frac{3\pi}{4}})\big\}}{|\Phi(re^{\iota\frac{3\pi}{4}})|^2}\cdot\frac{e^{-\frac{1}{\sqrt{2}}r}}{r}\d r, & \min\{\alpha_1,\alpha_2\}\ge1,
			\end{array}\right.
		\end{split}
	\end{equation*}
	where $\Phi(z)=\int_{\R}\frac{1}{z+T|v|^{\alpha_1}+T|v|^{\alpha_2}}\d v$ for $z\in\C\backslash\{z\le0\}$.
\end{theorem}
\begin{remark}
	Let $X^\alpha=\{X^\alpha_t,t\ge0\}$ be a symmetric $\alpha-$stable process with $\alpha\in(1,2]$. Using the same method, the small ball probability for local time $l(T)=\int_{0}^{T}\delta(X_t^\alpha)dt$ becomes 
	\begin{align*}
		\lim\limits_{\varepsilon\downarrow0}\varepsilon^{-1}P(l(T)\le\varepsilon)&=\frac{2}{T}\int_0^{+\infty}\frac{\Im\big\{ e^{\iota\frac{r}{\sqrt{2}}}\phi(re^{-\iota\frac{3\pi}{4}})\big\}}{|\phi(re^{\iota\frac{3\pi}{4}})|^2}\cdot\frac{e^{-\frac{1}{\sqrt{2}}r}}{r}\d r\\
		%&=\frac{\alpha T^{\frac1\alpha-1}}{\pi}\sin\frac{\pi}{\alpha} \int_0^{+\infty}\sin\Big(\frac{r}{\sqrt{2}}+\frac{3\pi}{4}(1-\frac1\alpha)\Big)\cdot\frac{e^{-\frac{1}{\sqrt{2}}r}}{r^{\frac1\alpha}}\d r
		&=\frac{\alpha T^{\frac1\alpha-1}}{\pi}\sin^2\frac{\pi}{\alpha}\cdot\Gamma(1-\frac{1}{\alpha}),
	\end{align*} 
	where $\phi(z)=\int_{\R}\frac{1}{z+T|v|^{\alpha}}\d v$ for $z\in\C\backslash\{z\le0\}$ and $\Gamma(\cdot)$ is the gamma function.
\end{remark}
As a direct consequence, we obtain the following integrability criterion. A direct application of small ball probability. Like \cite{cs2025}, there is the following corollary: 
\begin{corollary}
	The negative-moment-integrability 
	\[\E[\widetilde{L}(T)]^{-p}<+\infty\]
	exists if and only if $p<1$.
\end{corollary}

When $\alpha_1=\alpha_2=2$, there exists $\widetilde{L}(T)=\int_{0}^{T}\delta(B_t-\widetilde{B}_t)\d t$ and the integrability $$E\Big(\int_{0}^{T}\delta(B_t-\widetilde{B}_t)\d t\Big)^{-p}<+\infty$$
plays an important rule when studying the regularity of the probability distribution for the solution of 1-dimensional parabolic Anderson equation
\begin{align*}
 \frac{\partial}{\partial t}u(t,x)=\frac{\partial^2}{\partial x^2}u(t,x)+u(t,x)\cdot\dot{W}(t,x),\qquad t\ge0,x\in\R,   
\end{align*}
where $\dot{W}$ is the space-time Gaussian white noise with covariance 
\[\E(\dot{W}(t,x)\dot{W}(s,y))=\delta(t-s)\delta(x-y).\]
For more about this topic, \cite{hns2011} and references therein are recommended for readers.

The remainder of this paper is organized as follows. Section~2 collects preliminary lemmas concerning contour integrals and moment representations. Section~3 presents the proof of Theorem~1.1 in three logically connected steps. Section~4 establishes the technical estimates required for the contour analysis, with particular attention to the behavior of the key auxiliary function $\Phi(z)$. Finally, the proof of Proposition~1.1 is given in Section~5, where we use polynomial approximation to justify the Tauberian link between Laplace asymptotics and small ball probabilities.

\section{Preliminaries}
In this section, a contour which will play a key role in the proof of main theorem will be introduced first.
\begin{definition}\label{gamma}
	For any $R>0$, $\gamma(R,\frac{3\pi}{4})$ denotes a contour oriented by non-decreasing argument $\arg\zeta$ consisting of the following three parts: the ray $\arg\zeta=-\frac{3\pi}{4},|\zeta|\ge R$, the arc $-\frac{3\pi}{4}\le\arg\zeta\le \frac{3\pi}{4},|\zeta|= R$ and the ray $\arg\zeta=\frac{3\pi}{4},|\zeta|\ge R$.
\end{definition}
Here is a lemma about the integral representation of reciprocal gamma function using the above contour.

\begin{lemma}\label{reciprocal}
	For any $R>0$, $\gamma(R,\frac{3\pi}{4})$ defined in Definition \ref{gamma} and $s\in\C$, there exists equality 
	\begin{align}\label{Reciprocal}
		\frac{1}{\Gamma(s)}=\frac{1}{2\pi\iota}\int_{\gamma(R,\frac{3\pi}{4})}e^zz^{-s}\d z,
	\end{align}
	where $z^{-s}=e^{-s(\ln|z|+\iota\arg z)}$ is the principal branch of complex power function for $z\ne0$ with argument $\arg z\in(-\pi,\pi]$.
\end{lemma}
\begin{proof}[The proof of Lemma \ref{reciprocal}]
	For $0\le a,b,c\le+\infty$ and $\phi,\alpha,\beta\in\R$, denote paths 
	$$\gamma(a,b,\phi):\gamma(r)=re^{\iota \phi},r\text{ moves from }a\text{ to }b;$$ 
	$$\eta(c,\alpha,\beta):\eta(t)=ce^{\iota t},t\text{ moves from }\alpha\text{ to }\beta.$$
	
	For any given $R>0$, the path $\gamma(R,\frac{3\pi}{4})$ can be divided as 
	\begin{align*}
		\gamma(R,\frac{3\pi}{4})=\gamma(+\infty,R,-\frac{3\pi}{4})+\eta(R,-\frac{3\pi}{4},\frac{3\pi}{4})+\gamma(R,+\infty,\frac{3\pi}{4})
	\end{align*}
	Denote $J(R)=\int_{\gamma(R,\frac{3\pi}{4})}e^zz^{-s}\d z$. Using the division of path $\gamma(R,\frac{3\pi}{4})$, 
	\begin{align*}
		J(R)=\int_{\gamma(R,\frac{3\pi}{4})}e^zz^{-s}\d z&=\int_{\gamma(+\infty,R,-\frac{3\pi}{4})}e^zz^{-s}\d z+\int_{\eta(R,-\frac{3\pi}{4},\frac{3\pi}{4})}e^zz^{-s}\d z+\int_{\gamma(R,+\infty,\frac{3\pi}{4})}e^zz^{-s}\d z\\&:=J_1(R)+J_3(R)+J_2(R).
	\end{align*} Since $e^zz^{-s}$ is analytic on $\C\backslash\{z\le0\}$, using Cauchy's theorem, there exists 
	\[\int_{\gamma(R_1,\frac{3\pi}{4})}e^zz^{-s}\d z=\int_{\gamma(R_2,\frac{3\pi}{4})}e^zz^{-s}\d z\]
	for any $R_1>R_2>0$ because $$\gamma(R_1,\frac{3\pi}{4})-\gamma(R_2,\frac{3\pi}{4}):=\eta(R_1,-\frac{3\pi}{4},\frac{3\pi}{4})+\gamma(R_1,R_2,\frac{3\pi}{4})+\eta(R_1,\frac{3\pi}{4},-\frac{3\pi}{4})+\gamma(R_2,R_1,-\frac{3\pi}{4})$$
	is a closed path in $\C\backslash\{z\le0\}$. So that $J(R)$ is a constant function for $R>0$ and it only needs to consider the limit of $J(R)$ as $R\downarrow0$. 
	
	For $J_1(R)$, consider $z_1^{-s}=e^{-s(\ln|z|+\iota\arg_1 z)}$ for $z\ne0$ with argument $\arg_1 z\in(-\frac{3\pi}{2},\frac{\pi}{2}]$. Since $e^zz_1^{-s}$ is analytic on $\C\backslash\{\iota t:t\ge0\}$ and coincides with $e^zz^{-s}$ in $\gamma(+\infty,R,-\frac{3\pi}{4})$, by Cauchy's theorem, for $s\in\C$ with real part $\Re s<1$,
	\begin{align*}
		&\lim\limits_{R\downarrow0}J_1(R)\\&=\lim\limits_{R\downarrow0}\lim\limits_{R^*\to+\infty}\int_{\gamma(R^*,R,-\frac{3\pi}{4})}e^zz_1^{-s}\mathrm{d}z\\&=\lim\limits_{R\downarrow0}\lim\limits_{R^*\to+\infty}\Big(\int_{\gamma(R^*,R,-\pi)}e^zz_1^{-s}\d z+\int_{\eta(R^*,-\frac{3\pi}{4},-\pi)}e^zz_1^{-s}\d z+\int_{\eta(R,-\pi,-\frac{3\pi}{4})}e^zz_1^{-s}\d z\Big)\\&=\int_{\gamma(+\infty,0,-\pi)}e^zz_1^{-s}\d z-\iota\lim\limits_{R^*\to+\infty}\int_{-\pi}^{-\frac{3\pi}{4}}e^{R^*e^{\iota t}}R^{1-s}e^{\iota(1-s)t}\d t+\iota\lim\limits_{R\downarrow0}\int_{-\pi}^{-\frac{3\pi}{4}}e^{R e^{\iota t}}R^{1-s}e^{\iota(1-s)t}\d t
		\\&=e^{\iota s\pi}\int_{0}^{+\infty}e^{-r}r^{-s}\d r=e^{\iota s\pi}\Gamma(1-s)
	\end{align*}
	where the fourth equality comes from the fact that 
	\[\Big|e^{re^{\iota t}}r^{1-s}e^{\iota(1-s)r}\Big|\le r^{1-\Re s}e^{-\frac{\sqrt{2}}{2}r}\]
	for all $t\in[-\pi,-\frac{3\pi}{4}]$ and $r\in(0,+\infty)$. Similarly, consider $z_2^{-s}=e^{-s(\ln|z|+\iota\arg_2 z)}$ for $z\ne0$ with argument $\arg_2 z\in(-\frac{\pi}{2},\frac{3\pi}{2}]$, then for $s\in\C$ with real part $\Re s<1$,
	\begin{align*}
		&\lim\limits_{R\downarrow0}J_2(R)\\&=\lim\limits_{R\downarrow0}\lim\limits_{R^*\to+\infty}\int_{\gamma(R,R^*,\frac{3\pi}{4})}e^zz_2^{-s}\d z\\&=\lim\limits_{R\downarrow0}\lim\limits_{R^*\to+\infty}\Big(\int_{\gamma(R,R^*,\pi)}e^zz_2^{-s}\d z+\int_{\eta(R^*,\frac{3\pi}{4},\pi)}e^zz_2^{-s}\d z+\int_{\eta(R,\pi,\frac{3\pi}{4})}e^zz_2^{-s}\d z\Big)\\&=\int_{\gamma(0,+\infty,\pi)}e^zz_2^{-s}\d z+\iota\lim\limits_{R^*\to+\infty}\int_{\frac{3\pi}{4}}^{\pi}e^{R^*e^{\iota t}}(R^*)^{1-s}e^{\iota(1-s)t}\d t-\iota\lim\limits_{R\downarrow0}\int_{\frac{3\pi}{4}}^{\pi}e^{R e^{\iota t}}R^{1-s}e^{\iota(1-s)t}\d t
		\\&=-e^{-\iota s\pi}\int_{0}^{+\infty}e^{-r}r^{-s}\d r=-e^{-\iota s\pi}\Gamma(1-s).
	\end{align*}
	Moreover, $\lim\limits_{R\downarrow0}J_3(R)=0$ because of
	\begin{align*}
		|J_3(R)|=\Big|\int_{-\frac{3\pi}{4}}^{\frac{3\pi}{4}}e^{R e^{\iota t}}R^{1-s}e^{\iota(1-s)t}\d t\Big|\le\frac{3\pi}{2}R^{1-\Re s}e^{R}.
	\end{align*}
	Then using Euler's reflection formula $\Gamma(1-s)\Gamma(s)=\frac{\pi}{\sin\pi s}$ and the fact 
	\[\lim\limits_{R\downarrow0}J(R)=\lim\limits_{R\downarrow0}J_1(R)+\lim\limits_{R\downarrow0}J_2(R)+\lim\limits_{R\downarrow0}J_3(R)=e^{\iota s\pi}\Gamma(1-s)-e^{-\iota s\pi}\Gamma(1-s),\] for all $s\in\C$ with real part $\Re s<1$,
	\[\frac{1}{\Gamma(s)}=\frac{1}{2\pi\iota}\int_{\gamma(R,\frac{3\pi}{4})}e^zz^{-s}\d z.\]
	Because both sides of the above equality are entire functions and coincide with each other on $\Re s<1$, by uniqueness theorem of analytic function, they coincide on the complex plane, which completes the proof.
\end{proof}
Then a lemma which guarantees the change of contour integral and limitation will be proven.
\begin{lemma}\label{changeoforder}
	Given $R>0$ and $F_k(z):\C\to\C,k=0,1,2,\cdots$ complex functions satisfying that there exists a constant $C>0$ and $\beta>0$ s.t. 
	\[|F_k(z)|\le C\max\{|z|^{\beta},1\}\]
	for any $k\ge0$ and $z\in\gamma(R,\frac{3\pi}{4})$. If $F_k(z)\to F_0(z)$ as $k\to+\infty$, then  
	\[\lim_{k\to+\infty}\int_{\gamma(R,\frac{3\pi}{4})}F_k(z)e^z\d z=\int_{\gamma(R,\frac{3\pi}{4})}F_0(z)e^z\d z.\]
\end{lemma}
\begin{proof}
	Using notation $\gamma(a,b,\phi)$ and $\eta(c,\alpha,\beta)$ for $0\le a,b,c\le+\infty$ and $\phi,\alpha,\beta\in\R$ in Lemma \ref{reciprocal}, for any $k\ge0$, 
	\begin{align*}
			&\int_{\gamma(R,\frac{3\pi}{4})}F_k(z)e^z\d z\\&=\int_{\gamma(+\infty,R,-\frac{3\pi}{4})}F_k(z)\cdot e^z\d z+\int_{\eta(R,-\frac{3\pi}{4},\frac{3\pi}{4})}F_k(z)\cdot e^z\d z+\int_{\gamma(R,+\infty,\frac{3\pi}{4})}F_k(z)\cdot e^z\d z
			\\&=\int_{+\infty}^{R}\frac{F_k(re^{-\iota\frac{3\pi}{4}})}{r}\cdot e^{re^{-\iota\frac{3\pi}{4}}}\d r+\iota\int_{-\frac{3\pi}{4}}^{\frac{3\pi}{4}}F_k(Re^{\iota t})\cdot e^{Re^{\iota t}}\d t+\int_{R}^{+\infty}\frac{F_k(re^{\iota\frac{3\pi}{4}})}{r}\cdot e^{re^{\iota\frac{3\pi}{4}}}\d r.
	\end{align*}
	Because 
	\[\Big|\frac{F_k(re^{-\iota\frac{3\pi}{4}})}{r}\cdot e^{re^{-\iota\frac{3\pi}{4}}}\Big|\le\frac{C\max\{r^{\beta},1\}}{R}e^{-\frac{r}{\sqrt{2}}},\Big|\frac{F_k(re^{\iota\frac{3\pi}{4}})}{r}\cdot e^{re^{\iota\frac{3\pi}{4}}}\Big|\le\frac{C\max\{r^{\beta},1\}}{R}e^{-\frac{r}{\sqrt{2}}}\]
	for any $r\ge R$ and $\Big|F_k(Re^{\iota t})\cdot e^{Re^{\iota t}}\Big|\le C\max\{R^{\beta},1\}\cdot e^R$ for any $t\in[-\frac{3\pi}{4},\frac{3\pi}{4}]$,
	the desired result comes from the dominated convergence theorem.
\end{proof}
    Finally a lemma about $\gamma(R,\frac{3\pi}{4})$ will be given.
\begin{lemma}\label{simplex}
	Assume $n\ge1$, $R>0$, $T>0$ and  
	\[D^n_T=\{(u_1,u_2,\cdots,u_n):u_1+u_2+\cdots+u_n<1,u_j>0,j=1,2,\cdots,n\}.\] Then
	\begin{align*}
		\int_{D^n_T}\exp\Big(-\sum_{j=1}^{n}a_ju_j\Big)\d u_1\cdots \d u_n=\frac{T^n}{2\pi\iota}\int_{\gamma( R,\frac{3\pi}{4})}\prod_{j=1}^{n}\frac{1}{z+Ta_j}\cdot e^zz^{-1}\d z
	\end{align*}
	for any $a_j\ge0$, $j=1,2,\cdots,n$. 
\end{lemma}
\begin{proof}[The proof of Lemma \ref{simplex}]
	Recall the series expansion $e^x=\sum_{k=0}^{+\infty}\frac{x^k}{k!}$ for all $x\in\R$. There exists 
	\begin{align}\label{lma1_step1}
			&\int_{D^n_T}\exp\Big(-\sum_{j=1}^{n}a_ju_j\Big)\d u_1\cdots \d u_n\nonumber\\&=\sum_{m_1=0}^{\infty}\cdots\sum_{m_n=0}^{\infty}\prod_{j=1}^n\frac{(-a_j)^{m_j}}{m_j!}\int_{D^n_T}\prod_{j=1}^nu_j^{m_j}\d u_1\cdots \d u_n\nonumber
			%\\&=\sum_{m_1=0}^{\infty}\cdots\sum_{m_n=0}^{\infty}\prod_{j=1}^n\frac{(-a_j)^{m_j}}{m_j!}T^{\sum_{j=1}^{n}(m_j+1)}\frac{\prod_{j=1}^{n}\Gamma(m_j+1)}{\Gamma(\sum_{j=1}^n(m_j+1)+1)}
			\\&=T^n\sum_{m_1=0}^{\infty}\cdots\sum_{m_n=0}^{\infty}\frac{\prod_{j=1}^{n}(-Ta_j)^{m_j}}{\Gamma(n+\sum_{j=1}^nm_j+1)},
	\end{align}
	where Fubini's theorem and the fact 
	\[\int_{D^n_T}\sum_{m_1=0}^{\infty}\cdots\sum_{m_n=0}^{\infty}\prod_{j=1}^n\frac{(a_ju_j)^{m_j}}{m_j!}\d u_1\cdots \d u_n=\int_{D^n_T}\exp\Big(\sum_{j=1}^{n}a_ju_j\Big)\d u_1\cdots \d u_n<+\infty\]
	are used to get the the first equality and the second equality comes from Lemma \ref{Dirichlet}. To calculate the sum of reciprocal Gamma functions in \eqref{lma1_step1}, denote $R^{*}=1+T\sum_{j=1}^{n}|a_j|$. Then by Lemma \ref{reciprocal},
	\begin{align}\label{lma1_step2}
			&T^n\sum_{m_1=0}^{\infty}\cdots\sum_{m_n=0}^{\infty}\frac{\prod_{j=1}^{n}(-Ta_j)^{m_j}}{\Gamma(n+\sum_{j=1}^nm_j+1)}\nonumber\\&=T^n\sum_{m_1=0}^{\infty}\cdots\sum_{m_n=0}^{\infty}\frac{1}{2\pi\iota}\int_{\gamma(R,\frac{3\pi}{4})}\prod_{j=1}^{n}\Big(\frac{-Ta_j}{z}\Big)^{m_j}e^zz^{-n-1}\d z\nonumber\\&=T^n\sum_{m_1=0}^{\infty}\cdots\sum_{m_n=0}^{\infty}\frac{1}{2\pi\iota}\int_{\gamma(R^{*},\frac{3\pi}{4})}\prod_{j=1}^{n}\Big(\frac{-Ta_j}{z}\Big)^{m_j}e^zz^{-n-1}\d z\nonumber\\&=\frac{T^n}{2\pi\iota}\int_{\gamma(R^{*},\frac{3\pi}{4})}\prod_{j=1}^{n}\frac{1}{z+Ta_j}\cdot e^zz^{-1}\d z\nonumber\\&=\frac{T^n}{2\pi\iota}\int_{\gamma(R,\frac{3\pi}{4})}\prod_{j=1}^{n}\frac{1}{z+Ta_j}\cdot e^zz^{-1}\d z,
	\end{align}
	where the second and last equalities come from the Cauchy's theorem because functions  $\prod_{j=1}^{n}\big(\frac{-Ta_j}{z}\big)^{m_j}e^zz^{-n-1}$ and $\prod_{j=1}^{n}\frac{1}{z+Ta_j}\cdot e^zz^{-1}$ are analytic in the region between $\gamma(R,\frac{3\pi}{4})$ and $\gamma(R^{*},\frac{3\pi}{4})$, and the third equality comes from Lemma \ref{changeoforder} and the fact that 
	\[\sum_{m_1=0}^{k}\cdots\sum_{m_n=0}^{k}\prod_{j=1}^{n}\Big|\frac{-Ta_j}{z}\Big|^{m_j}\le\frac{\sum_{j=1}^{n}T|a_j|}{R^{*}}\le (R^{*})^n<+\infty\]
	 on $\gamma(R^{*},\frac{3\pi}{4})$ for any $k\ge0$.
\end{proof}
\section{Proof of main theorem}

The proof of this theorem will be divided into three steps.

{\bf Step I: Moment computation $\lim\limits_{\varepsilon\downarrow0}\E[\widetilde{L}_\varepsilon(T)]^m$.}

Our first goal is to obtain an explicit formula for the $m$-th moment of the regularized collision local time $\widetilde{L}_\varepsilon(T)$. This step relies on Fourier inversion and a change of variables that decouples the two stable processes.

  The limit of $\E[\widetilde{L}_\varepsilon(T)]^m$ as $\varepsilon\downarrow0$ will be calculated first.  Denote $$[0,T]^m_<=\{(t_1,t_2,\cdots,t_m):0=t_0<t_1<t_2<\cdots<t_m<T\}$$
  and $D^m_T$ as that in Lemma \ref{simplex}. 
  From \eqref{increment} and \eqref{fourier}, there exists
  
 \begin{align}
 		&\E[\widetilde{L}_\varepsilon(T)]^m\nonumber\\&=\frac{1}{(2\pi)^{m}}\E\Big(\int_{0}^{T}\int_{\R}e^{-\frac{\varepsilon}{2}|u|^2}e^{\iota u\cdot (X^{\alpha_1}_t-\widetilde{X}^{\alpha_2}_t)}\d u\d t\Big)^m\nonumber
 %	\\&=\frac{1}{(2\pi)^{m}}\int_{[0,T]^m}\int_{\R^{m}}e^{-\iota\sum_{j=1}^{m}u_j\cdot x-\frac{\varepsilon}{2}\sum_{j=1}^{m}|u_j|^2}\E e^{\iota \sum_{j=1}^{m}u_j\cdot X^{\alpha_1}_{t_j}}\E e^{-\iota \sum_{j=1}^{m}u_j\cdot\widetilde{X}^{\alpha_2}_{t_j}}dudt
 	\\&=\frac{m!}{(2\pi)^{m}}\int_{[0,T]^m_<}\int_{\R^{m}}e^{-\frac{\varepsilon}{2}\sum_{j=1}^{m}|u_j|^2}\E e^{\iota \sum_{j=1}^{m}u_j\cdot X^{\alpha_1}_{t_j}}\E e^{-\iota \sum_{j=1}^{m}u_j\cdot\widetilde{X}^{\alpha_2}_{t_j}}\d u\d t\nonumber
 	\\&=\frac{m!}{(2\pi)^{m}}\int_{[0,T]^m_<}\int_{\R^{m}}e^{-\frac{\varepsilon}{2}\sum_{j=1}^{m}|u_j|^2} e^{-\sum_{j=1}^{m}(|\sum_{i=j}^{m}u_j|^{\alpha_1}+|\sum_{i=j}^{m}u_j|^{\alpha_2})(t_j-t_{j-1})}\d u\d t\nonumber
 	\\&=\frac{m!}{(2\pi)^{m}}\int_{D^m_T}\int_{\R^{m}}e^{-\frac{\varepsilon}{2}\sum_{j=1}^{m}|v_j-v_{j+1}|^2}e^{-\sum_{j=1}^{m}(|v_j|^{\alpha_1}+|v_j|^{\alpha_2})s_j}\d v\d s,
 \end{align}
 where coordinate transforms $s_j=t_j-t_{j-1}$, $v_j=\sum_{i=j}^{m}u_i$, $j=1,2,\cdots,m$ and $v_{m+1}=0$ are used in the last equality. Using \eqref{dirichlet}, 
\begin{align}\label{bound}
		&\int_{D^m_T}\int_{\R^{m}}e^{-\sum_{j=1}^{m}(|v_j|^{\alpha_1}+|v_j|^{\alpha_2})s_j}\d v\d s\nonumber\\&=\int_{D^m_T}\prod_{j=1}^{m}s_j^{-\frac{1}{\alpha}}\d s\cdot\Big(\int_{\R}e^{-|v|^{\max\{\alpha_1,\alpha_2\}}}\d v\Big)^m\nonumber\\&=\frac{T^{m-\frac{m}{\max\{\alpha_1,\alpha_2\}}}}{\Gamma(m(1-\frac{1}{\max\{\alpha_1,\alpha_2\}})+1)}\Big(\Gamma\Big(1-\frac{1}{\max\{\alpha_1,\alpha_2\}}\Big)\int_{\R}e^{-|v|^{\max\{\alpha_1,\alpha_2\}}}\d v\Big)^m<+\infty.
\end{align}
 So that by the dominated convergence theorem and $$\big|e^{-\frac{\varepsilon}{2}\sum_{j=1}^{m}|v_j-v_{j+1}|^2}\big|\le1,$$
there exists 
\begin{align}\label{limit}
		\lim\limits_{\varepsilon\downarrow0}\E[\widetilde{L}_\varepsilon(T)]^m&=\frac{m!}{(2\pi)^{m}}\int_{D^m_T}\int_{\R^{m}}\lim\limits_{\varepsilon\downarrow0}e^{-\frac{\varepsilon}{2}\sum_{j=1}^{m}|v_j-v_{j+1}|^2}e^{-\sum_{j=1}^{m}(|v_j|^{\alpha_1}+|v_j|^{\alpha_2})s_j}\d v\d s\nonumber\\&=\frac{m!}{(2\pi)^{m}}\int_{D^m_T}\int_{\R^{m}}e^{-\sum_{j=1}^{m}(|v_j|^{\alpha_1}+|v_j|^{\alpha_2})s_j}\d v\d s\nonumber\\&=\frac{m!}{(2\pi)^{m}}\int_{\R^{m}}\int_{D^m_T}e^{-\sum_{j=1}^{m}(|v_j|^{\alpha_1}+|v_j|^{\alpha_2})s_j}\d s\d v.
\end{align}

{\bf Step II: Laplace transform $\E e^{-\lambda\widetilde{L}(T)}$ via contour integration.}

With the moment formula in hand, we consider the Laplace transform $\mathbb{E}[e^{-\lambda \widetilde{L}(T)}] = \sum_{m=0}^\infty \frac{(-\lambda)^m}{m!} \mathbb{E}[\widetilde{L}(T)^m]$. The key insight is to apply Lemma \ref{simplex} to convert this infinite series into a single contour integral involving the auxiliary function $\Phi(z) = \int_{\mathbb{R}} \frac{1}{z + T|v|^{\alpha_1} + T|v|^{\alpha_2}} \, \d v$. 

From the fact that $|e^{-x}-e^{-y}|\le|x-y|$ for all $x,y\ge0$, 
\[|\E e^{-\lambda\widetilde{L}(T)}-\E e^{-\lambda\widetilde{L}_\varepsilon(T)}|\le|\lambda|\E|\widetilde{L}_\varepsilon(T)-\widetilde{L}(T)|\le|\lambda|\Big(\E|\widetilde{L}_\varepsilon(T)-\widetilde{L}(T)|^2\Big)^{\frac12},\]
which means that $$\E e^{-\lambda\widetilde{L}(T)}=\lim\limits_{\varepsilon\downarrow0}\E e^{-\lambda\widetilde{L}_\varepsilon(T)}=\lim\limits_{\varepsilon\downarrow0}\sum_{m=0}^{+\infty}\frac{(-\lambda)^m}{m!}\E[\widetilde{L}_\varepsilon(T)]^m.$$
From \eqref{bound}, 
\begin{align*}
	\Big|\frac{(-\lambda)^m}{m!}\E[\widetilde{L}_\varepsilon(T)]^m\Big|\le\frac{c^m}{\Gamma(m(1-\frac{1}{\alpha_1})+1)},
\end{align*}
where $c=\frac{T^{1-\frac{1}{\alpha_1}}|\lambda|\Gamma(1-\frac{1}{\alpha_1})}{2\pi}\int_{\R}e^{-|v|^{\alpha_1}}\d v$. Using Stirling's formula $\lim\limits_{x\to+\infty}\frac{\Gamma(x+1)}{\sqrt{2\pi x}(\frac{x}{e})^x}=1$, it is easy to get that $\lim\limits_{m\to+\infty}\sqrt[m]{\frac{c^m}{\Gamma(m(1-\frac{1}{\alpha_1})+1)}}=0$ and then $\sum_{m=0}^{+\infty}\frac{c^m}{\Gamma(m(1-\frac{1}{\alpha_1})+1)}<+\infty$. So that using the dominated convergence theorem, \eqref{limit} and \eqref{dirichlet}, 
\begin{align}\label{mainproofstep2_1}
		\E e^{-\lambda\widetilde{L}(T)}&=\sum_{m=0}^{+\infty}\frac{(-\lambda)^m}{m!}\lim\limits_{\varepsilon\downarrow0}\E[\widetilde{L}_\varepsilon(T)]^m\nonumber
		\\&=\sum_{m=0}^{+\infty}\frac{(-\lambda)^m}{m!}\cdot\frac{m!}{(2\pi)^{m}}\int_{\R^{m}}\int_{D^m_T}e^{-\sum_{j=1}^{m}(|v_j|^{\alpha_1}+|v_j|^{\alpha_2})s_j}\d s\d v\nonumber
		\\&=\sum_{m=0}^{+\infty}\frac{(-\iota)(-\lambda T)^m}{(2\pi)^{m+1}}\int_{\R^{m}}\int_{\gamma(R,\frac{3\pi}{4})}\prod_{j=1}^{m}\frac{1}{z+T|v_j|^{\alpha_1}+T|v_j|^{\alpha_2}}\cdot e^zz^{-1}\d z\d v\nonumber
		\\&=\sum_{m=0}^{+\infty}\frac{(-\iota)(-\lambda T)^m}{(2\pi)^{m+1}}\int_{\gamma(R,\frac{3\pi}{4})}\int_{\R^{m}}\prod_{j=1}^{m}\frac{1}{z+T|v_j|^{\alpha_1}+T|v_j|^{\alpha_2}}dv\cdot e^zz^{-1}\d z\nonumber
		\\&=\sum_{m=0}^{+\infty}\frac{(-\iota)(-\lambda T)^m}{(2\pi)^{m+1}}\int_{\gamma(R,\frac{3\pi}{4})}\big(\Phi(z)\big)^{m}\cdot e^zz^{-1}\d z,
\end{align}
where $R>0$ is an arbitrary positive number, $\Phi(z)=\int_{\R}\frac{1}{z+T|v|^{\alpha_1}+T|v|^{\alpha_2}}\d v$ is defined in Lemma \ref{Phi} and the fourth equality is obtained from Lemma \ref{argument} and Fubini's theorem.

Notice that from Lemma \ref{Phi}, there exists  
\[|\Phi(z)|\le4\int_{\R}\frac{1}{1+T|v|^{\max\{\alpha_1,\alpha_2\}}}\d v\cdot |z|^{\frac{1}{\max\{\alpha_1,\alpha_2\}}-1}\]
%\[|\Phi_0(z)|\ge\int_{\R}\frac{1}{|z|+T|v|^{\alpha_1}+T|v|^{\alpha_2}}dv\ge\frac{(|z|+1)^{\frac{1}{\max\{\alpha_1,\alpha_2\}}-1}}{2T+1}\int_{\R}\frac{1}{1+|v|^{\max\{\alpha_1,\alpha_2\}}}dv\]
on $\gamma(R,\frac{3\pi}{4})$. Denote $$R^{**}=\Big(\frac{4\lambda T}{\pi}\int_{\R}\frac{1}{1+T|v|^{\max\{\alpha_1,\alpha_2\}}}\d v\Big)^{\frac{\max\{\alpha_1,\alpha_2\}}{\max\{\alpha_1,\alpha_2\}-1}},$$
where $\Big|\frac{-\lambda T\Phi(z)}{2\pi}\Big|\le\frac12<1$ when $|z|\ge R^{**}$ and $\arg z\in[-\frac{3\pi}{4},\frac{3\pi}{4}]$.
Then there exists
\begin{align}\label{mainproofstep2_2}
		&\sum_{m=0}^{+\infty}\frac{(-\iota)(-\lambda T)^m}{(2\pi)^{m+1}}\int_{\gamma(R,\frac{3\pi}{4})}\big(\Phi(z)\big)^{m}\cdot e^zz^{-1}\d z\nonumber
		\\&=\sum_{m=0}^{+\infty}\frac{1}{2\pi\iota}\int_{\gamma(R^{**},\frac{3\pi}{4})}\bigg(\frac{-\lambda T\Phi(z)}{2\pi}\bigg)^{m}\cdot e^zz^{-1}\d z\nonumber
		\\&=\frac{1}{2\pi\iota}\int_{\gamma(R^{**},\frac{3\pi}{4})}\frac{1}{1+\lambda T(2\pi)^{-1}\Phi(z)}\cdot e^zz^{-1}\d z\nonumber
		\\&=\frac{1}{2\pi\iota}\int_{\gamma(R,\frac{3\pi}{4})}\frac{1}{1+\lambda T(2\pi)^{-1}\Phi(z)}\cdot e^zz^{-1}\d z,
\end{align}
where the first equality comes from the fact that $\big(\Phi_0(z)\big)^{m}$ is analytic between $\gamma(R^{**},\frac{3\pi}{4})$, the second equality comes from Lemma \ref{changeoforder} and $\gamma(R,\frac{3\pi}{4})$ and the last equality comes from the fact that $\frac{1}{1+\lambda T(2\pi)^{-1}\Phi_0(z)}$ is also analytic in this region, which is obtained from \eqref{Phi2} and \eqref{Phi3} in Lemma \ref{Phi}. Combining \eqref{mainproofstep2_1} and \eqref{mainproofstep2_2}, there exists
\[\E e^{-\lambda\widetilde{L}(T)}=\frac{1}{2\pi\iota}\int_{\gamma(R,\frac{3\pi}{4})}\frac{1}{1+\lambda T(2\pi)^{-1}\Phi(z)}\cdot e^zz^{-1}\d z.\]

{\bf Step III: Asymptotic extraction and small ball limit property of $P(\widetilde{L}(T)\le\varepsilon)$. }

Finally, we analyze the behavior of the contour integral as $\lambda \to \infty$. By deforming the contour and applying dominated convergence (justified by Lemmas \ref{argument}--\ref{integralofPhi}), we compute $\lim_{\lambda \to \infty} \lambda \mathbb{E}[e^{-\lambda \widetilde{L}(T)}]$. Proposition \ref{Tauberian} immediately yields the desired small ball probability.

 Since $\arg \Phi(z)\in[-\frac{3\pi}{4},\frac{3\pi}{4}]$ when $\arg z\in[-\frac{3\pi}{4},\frac{3\pi}{4}]$ from \eqref{Phi3} of Lemma \ref{Phi}, using \eqref{argu1} of Lemma \ref{argument} and \eqref{Phi2} of Lemma \ref{Phi}, 
\begin{equation*}
	\begin{split}
		\bigg|\frac{\lambda}{1+\lambda T(2\pi)^{-1}\Phi(z)}\bigg|&\le\frac{4\lambda}{1+|\lambda T(2\pi)^{-1}\Phi(z)|}\\&\le\frac{4}{| T(2\pi)^{-1}\Phi(z)|}\\&\le\frac{32\pi}{T}\cdot\Big(\int_{\R}\frac{1}{1+T|v|^{\alpha_1}+T|v|^{\alpha_2}}\d v\Big)^{-1}\cdot\big(\max\{|z|,1\}\big)^{1-\frac{1}{\max\{\alpha_1,\alpha_2\}}}
	\end{split}
\end{equation*}
when $z\in\gamma(R,\frac{3\pi}{4})$. So using Lemma \ref{changeoforder}, 
\begin{equation*}
\begin{split}
		\lim\limits_{\lambda\to+\infty}\lambda\E e^{-\lambda\widetilde{L}(T)}&=\frac{1}{2\pi\iota}\int_{\gamma(R,\frac{3\pi}{4})}\lim\limits_{\lambda\to+\infty}\frac{\lambda}{1+\lambda T(2\pi)^{-1}\Phi(z)}\cdot e^zz^{-1}\d z\\&=\frac{1}{2\pi\iota}\int_{\gamma(R,\frac{3\pi}{4})}\frac{1}{T(2\pi)^{-1}\Phi(z)}\cdot e^zz^{-1}\d z\\&=\frac{1}{\iota T}\int_{\gamma(R,\frac{3\pi}{4})}\frac{1}{\Phi(z)}\cdot e^zz^{-1}\d z.
\end{split}
\end{equation*}
Since $\widetilde{L}(T)$ is a non-negative random variable, using Proposition \ref{Tauberian}, 
\[\lim\limits_{\varepsilon\downarrow0}\varepsilon^{-1}P(\widetilde{L}(T)\le\varepsilon)=\frac{1}{\iota T}\int_{\gamma(R,\frac{3\pi}{4})}\frac{1}{\Phi(z)}\cdot e^zz^{-1}\d z.\] 
Finally from Lemma \ref{integralofPhi}, the desired result is obtained.
\section{Technical Lemmas}
In this section, some technical lemmas are proven.
\begin{lemma}\label{argument}
	Given $z\in\C$ whose argument $\arg z\in[-\frac{3\pi}{4},\frac{3\pi}{4}]$ and $A\ge0$, there exists 
	\begin{equation}\label{argu1}
		\begin{split}
			\frac{1}{|z+A|}\le\frac{4}{|z|+A}.
		\end{split}
	\end{equation}
	Moreover, if $A$ is a function satisfying $A(v)\ge0,v\in\R$, 
	\begin{equation}\label{argu2}
		\begin{split}
			\Big|\int_{\R}\frac{1}{z+A(v)}\d v\Big|\ge\frac14\int_{\R}\frac{1}{|z|+A(v)}\d v.
		\end{split}
	\end{equation}
\end{lemma}
\begin{proof}
	Since $z=|z|e^{\iota\arg z}$, there exists $z+A=|z|\cos\arg z+A+\iota|z|\sin\arg z$ and so that 
	$$|z+A|=\sqrt{|z|^2+A^2+2|z|A\cos\arg z}\ge\sqrt{|z|^2+A^2-\sqrt{2}|z|A}$$
	from $\cos\arg z\ge-\frac{1}{\sqrt{2}}$ for $z\in[-\frac{3\pi}{4},\frac{3\pi}{4}]$. Using the fact $1-\frac{1}{\sqrt{2}}\ge\frac{1}{4}$, 
	\begin{equation*}
		\begin{split}
			|z|^2+A^2-\sqrt{2}|z|A&=(1-\frac{1}{\sqrt{2}})(|z|^2+A^2)+\frac{1}{\sqrt{2}}(|z|-A)^2\\
			&\ge\frac{1}{4}(|z|^2+A^2)\ge\frac{1}{8}(|z|^2+A^2+2|z|A)\ge\frac{1}{16}(|z|+A)^2,
		\end{split}
	\end{equation*}
	which means $|z+A|\ge\frac{1}{4}(|z|+A)$ and \eqref{argu1} is obtained. To prove \eqref{argu2}, using $z=|z|e^{\iota\arg z}$ again, 
	\begin{equation*}
	\begin{split}
		&\Big|\int_{\R}\frac{1}{z+A(v)}\d v\Big|^2\\=&\Big|\int_{\R}\frac{|z|\cos\arg z+A(v)}{|z|^2+A^2(v)+2|z|A(v)\cos\arg z}\d v-\iota\int_{\R}\frac{|z|\sin\arg z}{|z|^2+A^2(v)+2|z|A(v)\cos\arg z}\d v\Big|^2\\=&\int_{\R^{2}}\frac{|z|^2+A(v_1)A(v_2)+|z|\cos\arg z(A(v_1)+A(v_2))}{(|z|^2+A^2(v_1)+2|z|A(v_1)\cos\arg z)(|z|^2+A^2(v_2)+2|z|A(v_2)\cos\arg z)}\d v_1\d v_2.
	\end{split}
	\end{equation*}
	Since $\cos\arg z\in[-\frac{1}{\sqrt{2}},1]$ and $\frac{1}{\sqrt{2}}<\frac{7}{8}$, 
	\begin{equation*}
		\begin{split}
			&\Big|\int_{\R}\frac{1}{z+A(v)}\d v\Big|^2\\\ge&\int_{\R^{2}}\frac{|z|^2+A(v_1)A(v_2)-\frac{7}{8}|z|(A(v_1)+A(v_2))}{(|z|^2+A^2(v_1)+2|z|A(v_1)\cos\arg z)(|z|^2+A^2(v_2)+2|z|A(v_2)\cos\arg z)}\d v_1\d v_2\\=&\frac{1}{16}\Big(\int_{\R}\frac{|z|+A(v)}{|z|^2+A^2(v)+2|z|A(v)\cos\arg z}\d v\Big)^2+\frac{15}{16}\Big(\int_{\R}\frac{|z|-A(v)}{|z|^2+A^2(v)+2|z|A(v)\cos\arg z}\d v\Big)^2\\\ge&\frac{1}{16}\Big(\int_{\R}\frac{|z|+A(v)}{|z|^2+A^2(v)+2|z|A(v)}\d v\Big)^2=\frac{1}{16}\Big(\int_{\R}\frac{1}{|z|+A(v)}\d v\Big)^2,
		\end{split}
	\end{equation*}
	which completes the proof of \eqref{argu2}.
\end{proof}
\begin{lemma}\label{Phi}
	Given $z\in\C$, denote $\Phi(z)=\int_{\R}\frac{1}{z+T|v|^{\alpha_1}+T|v|^{\alpha_2}}\d v$. If the argument of $z$ satisfies $\arg z\in[-\frac{3\pi}{4},\frac{3\pi}{4}]$, there exist the following results:	
	\begin{equation}\label{Phi1}
		\begin{split}
			|\Phi(z)|\le\int_{\R}\frac{4}{1+T|v|^{\max\{\alpha_1,\alpha_2\}}}\d v\cdot |z|^{\frac{1}{\max\{\alpha_1,\alpha_2\}}-1},
		\end{split}
	\end{equation}
	\begin{equation}\label{Phi2}
		\begin{split}
			|\Phi(z)|\ge\frac14\int_{\R}\frac{1}{1+T|v|^{\alpha_1}+T|v|^{\alpha_2}}\d v\cdot \max\{|z|^{\frac{1}{\max\{\alpha_1,\alpha_2\}}-1},1\}
		\end{split}
	\end{equation}
	and
	\begin{equation}\label{Phi3}
		\begin{split}
			\cos\arg \Phi(z)\ge-\frac{\sqrt{2}}{2}.
		\end{split}
	\end{equation}
\end{lemma}
\begin{proof}
	Firstly, from \eqref{argu1}, 
	\begin{equation*}
		\begin{split}
			|\Phi(z)|&\le\int_{\R}\frac{4}{|z|+T|v|^{\alpha_1}+T|v|^{\alpha_2}}\d v\\&\le\int_{\R}\frac{4}{|z|+T|v|^{\max\{\alpha_1,\alpha_2\}}}\d v\\&=\int_{\R}\frac{4}{1+T|v|^{\max\{\alpha_1,\alpha_2\}}}\d v\cdot |z|^{\frac{1}{\max\{\alpha_1,\alpha_2\}}-1}.
		\end{split}
	\end{equation*}
	Secondly, from \eqref{argu2}, $	|\Phi(z)|\ge\frac14\int_{\R}\frac{1}{1+T|v|^{\alpha_1}+T|v|^{\alpha_2}}\d v$ when $|z|\le1$ and 
	\begin{equation*}
		\begin{split}
			|\Phi(z)|&\ge\frac14\int_{\R}\frac{1}{|z|+T|v|^{\alpha_1}+T|v|^{\alpha_2}}\d v
			\\&\ge\frac14\int_{\R}\frac{1}{|z|+T|z|^{\frac{|\alpha_1-\alpha_2|}{\max\{\alpha_1,\alpha_2\}}}|v|^{\min\{\alpha_1,\alpha_2\}}+T|v|^{\max\{\alpha_1,\alpha_2\}}}\d v
			\\&=\frac14\int_{\R}\frac{1}{1+T|v|^{\alpha_1}+T|v|^{\alpha_2}}\d v\cdot |z|^{\frac{1}{\max\{\alpha_1,\alpha_2\}}-1}
		\end{split}
	\end{equation*}
	when $|z|\ge1$. So that 
	\begin{equation*}
		\begin{split}
			|\Phi(z)|&\ge\frac14\int_{\R}\frac{1}{1+T|v|^{\alpha_1}+T|v|^{\alpha_2}}\d v\cdot \min\{|z|^{\frac{1}{\max\{\alpha_1,\alpha_2\}}-1},1\}\\&=\frac14\int_{\R}\frac{1}{1+T|v|^{\alpha_1}+T|v|^{\alpha_2}}\d v\cdot \max\{|z|^{\frac{1}{\max\{\alpha_1,\alpha_2\}}-1},1\}.
		\end{split}
	\end{equation*}
	Finally, denote $$R(z)=\int_{\R}\frac{|z|\cos\arg z+T|v|^{\alpha_1}+T|v|^{\alpha_2}}{|z|^2+(T|v|^{\alpha_1}+T|v|^{\alpha_2})^2+2|z|(T|v|^{\alpha_1}+T|v|^{\alpha_2})\cos\arg z}\d v$$ and $$I(z)=\int_{\R}\frac{|z|\sin\arg z}{|z|^2+(T|v|^{\alpha_1}+T|v|^{\alpha_2})^2+2|z|(T|v|^{\alpha_1}+T|v|^{\alpha_2})\cos\arg z}\d v$$
	so that 
	\begin{equation*}
		\begin{split}
			\cos\arg \Phi(z)=\frac{R(z)}{\sqrt{R^2(z)+I^2(z)}}.
		\end{split}
	\end{equation*}
	When $R(z)\ge0$, it is easy to get $\cos\arg \Phi(z)\ge0>-\frac{\sqrt{2}}{2}$. When $R(z)<0$, if $\cos\arg z\ge0$, there exists $|z|\cos\arg z+A(v)\ge0$ and in this case $R(z)\ge0$, which contradicts with the assumption $R(z)<0$. Then there must be $-\frac1{\sqrt{2}}\le\cos\arg z<0$ and $\big|\frac{\cos\arg z}{\sin\arg z}\big|\le1$ because $\arg z\in[-\frac{3\pi}{4},\frac{3\pi}{4}]$, which means 
	$$-|I(z)|\le-\int_{\R}\frac{|z||\cos\arg z|}{|z|^2+(T|v|^{\alpha_1}+T|v|^{\alpha_2})^2+2|z|(T|v|^{\alpha_1}+T|v|^{\alpha_2})\cos\arg z}\d v\le R(z)<0.$$
	So   
	\begin{equation*}
		\begin{split}
			\cos\arg \Phi(z)=-\frac{1}{\sqrt{1+I^2(z)/R^2(z)}}\ge-\frac{1}{\sqrt{2}},
		\end{split}
	\end{equation*}
	which completes the proof.
\end{proof}

\begin{lemma}\label{integralofPhi}
	Assume $\alpha_1,\alpha_2>0$ satisfying $\max\{\alpha_1,\alpha_2\}>1$. For $\Phi(z)=\int_{\R}\frac{1}{z+T|v|^{\alpha_1}+T|v|^{\alpha_2}}\d v$ and $R>0$, there exists  
	\begin{equation*}
		\begin{split}
			&\frac{1}{\iota T}\int_{\gamma(R,\frac{3\pi}{4})}\frac{1}{\Phi(z)}\cdot e^zz^{-1}\d z\\&=\left\{\begin{array}{ll}
				\displaystyle\frac{2}{T}\int_0^{+\infty}\frac{\Im\big\{ e^{\iota\frac{r}{\sqrt{2}}}\Phi(re^{-\iota\frac{3\pi}{4}})\big\}}{|\Phi(re^{\iota\frac{3\pi}{4}})|^2}\cdot\frac{e^{-\frac{1}{\sqrt{2}}r}}{r}\d r+\frac{3\pi}{2\int_{\R}\frac{1}{|v|^{\alpha_1}+|v|^{\alpha_2}}\d v}, & \min\{\alpha_1,\alpha_2\}<1, \\\\\displaystyle\frac{2}{T}\int_0^{+\infty}\frac{\Im\big\{ e^{\iota\frac{r}{\sqrt{2}}}\Phi(re^{-\iota\frac{3\pi}{4}})\big\}}{|\Phi(re^{\iota\frac{3\pi}{4}})|^2}\cdot\frac{e^{-\frac{1}{\sqrt{2}}r}}{r}\d r, & \min\{\alpha_1,\alpha_2\}\ge1,
			\end{array}\right.
		\end{split}
	\end{equation*}
	where $\Im z$ is the imaginary part of complex number $z\in\C$.
\end{lemma}
\begin{proof}
	Firstly, given $R_2>R_1>0$, $\frac{1}{\Phi(z)}$ is analytic between $\gamma(R_1,\frac{3\pi}{4})$ and $\gamma(R_2,\frac{3\pi}{4})$ from lemma \ref{Phi}. Then by Cauchy's theorem,  $\int_{\gamma(R,\frac{3\pi}{4})}\frac{1}{\Phi(z)}\cdot e^zz^{-1}\d z$ is a constant when $R>0$. Using notation $\gamma(a,b,\phi)$ and $\eta(c,\alpha,\beta)$ for $0\le a,b,c\le+\infty$ and $\phi,\alpha,\beta\in\R$ in Lemma \ref{reciprocal}, there exists 
	\begin{equation*}
		\begin{split}
			&\int_{\gamma(R,\frac{3\pi}{4})}\frac{1}{\Phi(z)}\cdot e^zz^{-1}\d z\\&=\int_{\gamma(+\infty,R,-\frac{3\pi}{4})}\frac{1}{\Phi(z)}\cdot e^zz^{-1}\d z+\int_{\eta(R,-\frac{3\pi}{4},\frac{3\pi}{4})}\frac{1}{\Phi(z)}\cdot e^zz^{-1}\d z+\int_{\gamma(R,+\infty,\frac{3\pi}{4})}\frac{1}{\Phi(z)}\cdot e^zz^{-1}\d z
			\\&=\int_{+\infty}^{R}\frac{1}{\Phi(re^{-\iota\frac{3\pi}{4}})}\cdot e^{re^{-\iota\frac{3\pi}{4}}}\frac{1}{r}\d r+\iota\int_{-\frac{3\pi}{4}}^{\frac{3\pi}{4}}\frac{1}{\Phi(Re^{\iota t})}\cdot e^{Re^{\iota t}}\d t+\int_{R}^{+\infty}\frac{1}{\Phi(re^{\iota\frac{3\pi}{4}})}\cdot e^{re^{\iota\frac{3\pi}{4}}}\frac{1}{r}\d r\\&=\int_{R}^{+\infty}\frac{e^{\iota\frac{r}{\sqrt{2}}}\Phi(re^{-\iota\frac{3\pi}{4}})-e^{-\iota\frac{r}{\sqrt{2}}}\Phi(re^{\iota\frac{3\pi}{4}})}{|\Phi(re^{\iota\frac{3\pi}{4}})|^2}\cdot e^{-\frac{r}{\sqrt{2}}}\frac{1}{r}\d r+\iota\int_{-\frac{3\pi}{4}}^{\frac{3\pi}{4}}\frac{1}{\Phi(Re^{\iota t})}\cdot e^{R\cos t+\iota R\sin t}\d t\\&=2\iota\int_R^{+\infty}\frac{\Im\big\{ e^{\iota\frac{r}{\sqrt{2}}}\Phi(re^{-\iota\frac{3\pi}{4}})\big\}}{|\Phi(re^{\iota\frac{3\pi}{4}})|^2}\cdot\frac{e^{-\frac{1}{\sqrt{2}}r}}{r}\d r+\iota\int_{-\frac{3\pi}{4}}^{\frac{3\pi}{4}}\frac{1}{\Phi(Re^{\iota t})}\cdot e^{R\cos t+\iota R\sin t}\d t\\&=:2\iota\int_R^{+\infty}\bar{J}_1(r)\d r+\iota\int_{-\frac{3\pi}{4}}^{\frac{3\pi}{4}}\bar{J}_2(R,t)\d t.
		\end{split}
	\end{equation*}
	From \eqref{argu2} of Lemma \ref{argument}, there exists 
	\[|\bar{J}_2(R,t)|\le4\Big(\int_{\R}\frac{1}{R+T|v|^{\alpha_1}+T|v|^{\alpha_2}}\d v\Big)^{-1} e^{R}\le\Big(\int_{\R}\frac{1}{1+T|v|^{\alpha_1}+T|v|^{\alpha_2}}\d v\Big)^{-1} e\]
	 when $R\le1$ and $t\in[-\frac{3\pi}{4},\frac{3\pi}{4}]$. So using the dominated convergence theorem and the fact for any $t\in[-\frac{3\pi}{4},\frac{3\pi}{4}]$,
	 \begin{align}\label{limitofPhi}
	 	\lim\limits_{R\downarrow0}\Phi(Re^{\iota t})=\begin{cases}
	 		\displaystyle+\infty, & \min\{\alpha_1,\alpha_2\}\ge1,\max\{\alpha_1,\alpha_2\}>1,\\\\
	 		\displaystyle\int_{\R}\frac{1}{T|v|^{\alpha_1}+T|v|^{\alpha_2}}dv, & \min\{\alpha_1,\alpha_2\}<1,\max\{\alpha_1,\alpha_2\}>1,
	 	\end{cases}
	 \end{align}
	 there exists 
	 \begin{equation*}
	 	\begin{split}
	 		&\lim\limits_{R\downarrow0}\int_{-\frac{3\pi}{4}}^{\frac{3\pi}{4}}\frac{1}{\Phi(Re^{\iota t})}\cdot e^{R\cos t+\iota R\sin t}\d t\\&=\begin{cases}
	 			\displaystyle0, & \min\{\alpha_1,\alpha_2\}\ge1,\max\{\alpha_1,\alpha_2\}>1,\\\\
	 			\displaystyle\frac{3\pi}{2\int_{\R}\frac{1}{T|v|^{\alpha_1}+T|v|^{\alpha_2}}\d v}, & \min\{\alpha_1,\alpha_2\}<1,\max\{\alpha_1,\alpha_2\}>1.
	 		\end{cases}
	 	\end{split}
	 \end{equation*}
	 Moreover, since
	 \begin{equation*}
	 	\begin{split}
	 		\Big|\Im\big\{ e^{\iota\frac{r}{\sqrt{2}}}\Phi(re^{-\iota\frac{3\pi}{4}})\big\}\Big|&=\Big|\iota\sin\frac{r}{\sqrt{2}}\Big(\Phi(re^{-\iota\frac{3\pi}{4}})+\Phi(re^{\iota\frac{3\pi}{4}})\Big)+\cos\frac{r}{\sqrt{2}}\Big(\Phi(re^{-\iota\frac{3\pi}{4}})-\Phi(re^{\iota\frac{3\pi}{4}})\Big)\Big|\\&\le \sqrt{2}r|\Phi(re^{\iota\frac{3\pi}{4}})|+\int_{\R}\frac{\sqrt{2}r}{r^2-\sqrt{2}rT(|v|^{\alpha_1}+|v|^{\alpha_2})+T^2(|v|^{\alpha_1}+|v|^{\alpha_2})^2}\d v\\&\le \sqrt{2}r|\Phi(re^{\iota\frac{3\pi}{4}})|+\int_{\R}\frac{2(\sqrt{2}+1)r}{r^2+T^2(|v|^{\alpha_1}+|v|^{\alpha_2})^2}\d v,
	 	\end{split}
	 \end{equation*}
	 there exists 
	 \begin{equation}\label{upperJ1}
	 	\begin{split}
	 		\big|\bar{J}_1(r)\big|\le\frac{\sqrt{2}e^{-\frac{1}{\sqrt{2}}r}}{|\Phi(re^{\iota\frac{3\pi}{4}})|}+\frac{2(\sqrt{2}+1)e^{-\frac{1}{\sqrt{2}}r}}{|\Phi(re^{\iota\frac{3\pi}{4}})|^2}\int_{\R}\frac{1}{r^2+T^2(|v|^{\alpha_1}+|v|^{\alpha_2})^2}\d v.
	 	\end{split}
	 \end{equation}
	 From \eqref{argu2}, letting $\alpha^*=\min\{\alpha_1,\alpha_2\}$, there exists 
	 \begin{align}\label{upperJ2}
	 		|\Phi(re^{\iota\frac{3\pi}{4}})|&\ge\frac14\int_{\R}\frac{1}{r+T|v|^{\alpha_1}+T|v|^{\alpha_2}}\d v\nonumber\\&\ge\frac14\int_{0}^1\frac{1}{r+2Tv^{\alpha^*}}\d v\nonumber\\&\ge\frac18\int_{0}^1\frac{1}{(r^{\frac{1}{\alpha^*}}+(2T)^{\frac{1}{\alpha^*}}v)^{\alpha^*}}\d v\nonumber\\&=\frac{1}{8(2T)^{\frac{1}{\alpha^*}}}\times\begin{cases}
	 		\frac{1}{1-\alpha^*}(r^{\frac{1}{\alpha^*}}+(2T)^{\frac{1}{\alpha^*}})^{1-\alpha^*}+\frac{1}{\alpha^*-1}(\frac{1}{r})^{1-\frac{1}{\alpha^*}}	, & \alpha^*\ne1,\\
	 		\ln(r+2T)+\ln\frac{1}{r}	, & \alpha^*=1.
	 		\end{cases}
	 \end{align}
	 and 
	 \begin{align}\label{upperJ3}
	 		\int_{\R}\frac{1}{r^2+T^2(|v|^{\alpha_1}+|v|^{\alpha_2})^2}\d v&\le\begin{cases}
	 			\int_{\R}\frac{1}{T^2(|v|^{\alpha_1}+|v|^{\alpha_2})^2}\d v,	& \alpha^*<\frac{1}{2},\\
	 			2\int_{0}^{1}\frac{1}{r^2+T^2v}\d v+\int_{|v|\ge1}\frac{1}{T^2|v|^{2}}\d v	& \alpha^*=\frac12,\\
	 			\int_{\R}\frac{1}{r^2+T^2|v|^{2\alpha^*}}\d v& \alpha^*>\frac{1}{2}.
	 		\end{cases}\nonumber\\&\le\begin{cases}
	 		\int_{\R}\frac{1}{T^2(|v|^{\alpha_1}+|v|^{\alpha_2})^2}\d v,	& \alpha^*<\frac{1}{2},\\
	 		\frac{2}{T^2}\ln(r^2+T^2)+\frac{4}{T^2}\ln(\frac{1}{r})+\int_{|v|\ge1}\frac{1}{T^2|v|^{2}}\d v	& \alpha^*=\frac12,\\
	 		\int_{\R}\frac{1}{1+T^2|v|^{2\alpha^*}}\d v\cdot (\frac{1}{r})^{2-\frac{1}{\alpha^*}}	& \alpha^*>\frac{1}{2}.
	 		\end{cases}
	 \end{align}
	Then from \eqref{upperJ1}, \eqref{upperJ2} and \eqref{upperJ3}, there exists 
	\[\int_0^{+\infty}|\bar{J}_1(r)|\d r<+\infty\]
	and then 
	\[\lim\limits_{R\downarrow0}\int_R^{+\infty}\bar{J}_1(r)\d r=\int_0^{+\infty}\bar{J}_1(r)\d r,\]
	which completes the proof of this lemma.
\end{proof}
\begin{lemma}\label{Dirichlet}
	Assume that $n\ge1$ and $p_1,p_2,\cdots,p_n>0$. Then for $n-$dimensional simplex $D^n_T$ defined in Lemma \ref{simplex}, there exists 
	\begin{align}\label{dirichlet}
		\int_{D^n_T}\prod_{j=1}^{n}u_j^{p_j-1}\d u_1\d u_2\cdots \d u_n=T^{\sum_{j=1}^{n}p_j}\frac{\prod_{j=1}^n\Gamma(p_j)}{\Gamma(\sum_{j=1}^{n}p_j+1)}.
	\end{align}
\end{lemma}
\begin{proof}
	First for $n=1$, 
	\[\int_{0}^{T}u_1^{p_1-1}\d u_1=\frac{1}{p_1}=T^{p_1}\frac{\Gamma(p_1)}{\Gamma(p_1+1)}.\]
	If for $n=k(k\ge1)$, there exists
	\[\int_{D^k_T}\prod_{j=1}^{k}u_j^{p_j-1}\d u_1\d u_2\cdots \d u_k=T^{\sum_{j=1}^{k}p_j}\frac{\prod_{j=1}^k\Gamma(p_j)}{\Gamma(\sum_{j=1}^{k}p_j+1)},\]
	then for $n=k+1$, 
	\begin{align*}
		\int_{D^{k+1}_T}\prod_{j=1}^{k+1}u_j^{p_j-1}\d u_1\d u_2\cdots \d u_{k+1}=&\int_{0}^{T}u_{k+1}^{p_{k+1}-1}\Big(\int_{D^{k}_{1-u_{k+1}}}\prod_{j=1}^{k}u_j^{p_j-1}\d u_1\d u_2\cdots \d u_k\Big)\d u_{k+1}\\
		=&\int_{0}^{T}u_{k+1}^{p_{k+1}-1}(T-u_{k+1})^{\sum_{j=1}^{k}p_j}\frac{\prod_{j=1}^k\Gamma(p_j)}{\Gamma(\sum_{j=1}^{k}p_j+1)}\d u_{k+1}\\
		%=&T^{\sum_{j=1}^{k+1}p_j}\frac{\prod_{j=1}^k\Gamma(p_j)}{\Gamma(\sum_{j=1}^{k}p_j+1)}\int_{0}^{1}u_{k+1}^{p_{k+1}-1}(1-u_{k+1})^{\sum_{j=1}^{k}p_j}du_{k+1}\\
		=&T^{\sum_{j=1}^{k+1}p_j}\frac{\prod_{j=1}^k\Gamma(p_j)}{\Gamma(\sum_{j=1}^{k}p_j+1)}B(p_{k+1},\sum_{j=1}^{k}p_j+1)\\
		%=&T^{\sum_{j=1}^{k+1}p_j}\frac{\prod_{j=1}^k\Gamma(p_j)}{\Gamma(\sum_{j=1}^{k}p_j+1)}\frac{\Gamma(p_{k+1})\Gamma(\sum_{j=1}^{k}p_j+1)}{\Gamma(\sum_{j=1}^{k+1}p_j+1)}\\
		=&T^{\sum_{j=1}^{k+1}p_j}\frac{\prod_{j=1}^{k+1}\Gamma(p_j)}{\Gamma(\sum_{j=1}^{k+1}p_j+1)},
	\end{align*}
	which means that \eqref{dirichlet} exists for any $n\ge1$.
\end{proof}

%\section{Appendix}
\section{The proof of Proposition \ref{Tauberian}}
To establish the Tauberian link between the Laplace transform and the small ball probability, 
we  use the  Weierstrass approximation theorem to approximate the indicator function by polynomials. The details proof is as follows.

\begin{proof}[The proof of Proposition \ref{Tauberian}]
	Consider a function $g:[0,1]\to\R$ defined by 
	\[g(x)=\begin{cases}
		0, & x\in[0,e^{-1}) \\	1, & x\in[e^{-1},1]
	\end{cases}.\]
	Because $\P (X\le\frac{1}{\lambda})=\E g(e^{-\lambda X})$, it suffices to prove $$\lim\limits_{\lambda\to\infty}\lambda^{\alpha}\E g(e^{-\lambda X})=\frac{A}{\Gamma(1+\alpha)}.$$
	From \eqref{tauberian}, $$\lim\limits_{\lambda\to+\infty}\lambda^{\alpha}Ee^{-\lambda k X}=Ak^{-\alpha}=\frac{A}{\Gamma(\alpha)}\int_0^{\infty}v^{\alpha-1}e^{-kv}\d v=\frac{A}{\Gamma(1+\alpha)}\int_0^{\infty}e^{-kv^{1/{\alpha}}}\d v$$
	for any integer $k\ge1$, so that given an integer $m\ge1$ and a polynomial $$H(x)=\sum\limits_{k=1}^{m}a_kx^k,$$ there exists 
	\begin{align}\label{polynomial}
		\lim\limits_{\lambda\to+\infty}\lambda^{\alpha}\E H(e^{-\lambda X})=\frac{A}{\Gamma(1+\alpha)}\int_0^{\infty}H\big(e^{-v^{1/\alpha}}\big)\d v.
	\end{align}
	For any $\varepsilon\in(0,e^{-1})$, let 
	\begin{align}
		h_1(x)=\begin{cases}
			0, & x\in[0,e^{-1}-\varepsilon]\\
			\varepsilon^{-1}(x-e^{-1})+1, & x\in(e^{-1}-\varepsilon,e^{-1})\\1, & x\in[e^{-1},1]
		\end{cases}.
	\end{align}
	Because $\frac{h_1(x)}{x}+\varepsilon$ is continuous on $[0,1]$,
	using Weierstrass's approximation theorem, there exists a polynomial $p_1(x)$, s.t. 
	\[\Big|p_1(x)-\varepsilon-\frac{h_1(x)}{x}\Big|<\varepsilon\] and then
	\[\frac{h_1(x)}{x}<p_1(x)<2\varepsilon+\frac{h_1(x)}{x}\]
	for any $x\in[0,1]$. So that 
	\begin{align*}
		\lambda^{\alpha}\E g(e^{-\lambda X})\le\lambda^{\alpha}\E h_1(e^{-\lambda X})\le\lambda^{\alpha}\E \Big(e^{-\lambda X}p_1(e^{-\lambda X})\Big)
	\end{align*}
	and then 
	\begin{align*}
		\limsup_{\lambda\to+\infty}\lambda^{\alpha}\E g(e^{-\lambda X})&\le\frac{A}{\Gamma(1+\alpha)}\int_0^{\infty}e^{-v^{1/\alpha}}p_1\big(e^{-v^{1/\alpha}}\big)\d v\\&\le\frac{A}{\Gamma(1+\alpha)}\Big(\int_0^{\infty}h_1\big(e^{-v^{1/\alpha}}\big)\d v+2\varepsilon\int_0^{\infty}e^{-v^{1/\alpha}}\d v\Big)\\&=\frac{A}{\Gamma(1+\alpha)}\Big(1+\int_{1}^{\big(\ln(e^{-1}-\varepsilon)^{-1}\big)^{\alpha}}h_1\big(e^{-v^{1/\alpha}}\big)\d v+2\varepsilon\Gamma(1+\alpha)\Big)\\&\le\frac{A}{\Gamma(1+\alpha)}\Big(\big(\ln(e^{-1}-\varepsilon)^{-1}\big)^{\alpha}+2\varepsilon\Gamma(1+\alpha)\Big),
	\end{align*}
	where the equality comes from the fact that $$\int_0^{\infty}g\big(e^{-v^{1/\alpha}}\big)\d v=\int_0^{1}1\d v=1.$$ Letting $\varepsilon\downarrow0$, the inequality becomes 
	\[\limsup_{\lambda\to+\infty}\lambda^{\alpha}\E g(e^{-\lambda X})\le\frac{A}{\Gamma(1+\alpha)}.\]
	For any $\varepsilon\in(0,1-e^{-1})$, let 
	\begin{align}
		h_2(x)=\begin{cases}
			0, & x\in[0,e^{-1}]\\
			\varepsilon^{-1}(x-e^{-1}), & x\in(e^{-1},e^{-1}+\varepsilon)\\1, & x\in[e^{-1}+\varepsilon,1]
		\end{cases}
	\end{align}
	and there exists a polynomial $p_2(x)$ satisfying 
	\[\frac{h_2(x)}{x}-2\varepsilon<p_2(x)<\frac{h_2(x)}{x}\] 
	using Weierstrass's approximation theorem. Then the inequality  \[\liminf_{\lambda\to+\infty}\lambda^{\alpha}\E g(e^{-\lambda X})\ge\frac{A}{\Gamma(1+\alpha)}\]
	can be obtained similarly, which completes the proof.
\end{proof}

\bigskip
\subsection*{Acknowledgements}
The authors would like to thank Professor Fangjun Xu for for his valuable comments and suggestions on an earlier version of this manuscript, which significantly improved the quality of this work.

\subsection*{Funding}
Qian Yu is supported by National Natural Science Foundation of China (12201294). Minhao Hong was supported by the National Social Science Fund of China (25BTJ047).

\subsection*{Data Availability Statements} Data sharing not applicable to this article as no datasets were generated or analysed during the current study.

\subsection*{Declaration of interests} The authors declare that they have no known competing financial interests or personal relationships that
could have appeared to influence the work reported in this paper.

\end{document}